\documentclass{article}
\usepackage{amssymb,amsmath,amsthm,graphicx,ucs,hyperref}
\usepackage[utf8x]{inputenc}
\usepackage[english]{babel}

\newtheorem{theorem}{Theorem}
\newtheorem{corollary}{Corollary}
\newtheorem{proposition}{Proposition}
\newtheorem{lemma}{Lemma}
\newtheorem{definition}{Definition}

\title{
Weak Local Rules for Planar Octagonal Tilings
\thanks{This work was supported by the ANR project QuasiCool (ANR-12-JS02-011-01)}}

\author{
Nicolas Bédaride
\footnote{Aix Marseille Univ., CNRS, Centrale Marseille, I2M, UMR 7373, 13453 Marseille, France.}
\and
Thomas Fernique
\footnote{Univ. Paris 13, CNRS, Sorbonne Paris Cité, UMR 7030, 93430 Villetaneuse, France.}
}
\date{}

\begin{document}

\maketitle

\begin{abstract}
We provide an {\em effective} characterization of the {\em planar octagonal tilings} which admit {\em weak local rules}.
As a corollary, we show that they are all based on quadratic irrationalities, as conjectured by Thang Le in the 90's.
\end{abstract}

\section{Introduction}
\label{sec:introduction}

Non-periodic tilings have received a lot of attention since the discovery of quasicrystals in the early 80s, because they provide a model of their structure.
Two prominent methods to define non-periodic tilings are {\em substitutions} and {\em cut and projection} (for a general introduction to these methods, see, {e.g.}, \cite{Baake-Grimm-2013,Senechal-1995}).
However, to model the stabilization of quasicrystals by short-range energetic interaction, a crucial point is to know whether such non-periodic tilings admit {\em local rules}, that is, can be characterized by their patterns of a given size.\\

If one allows tiles to be {\em decorated}, then the tilings obtained by substitutions are known to (generally) admit local rules (see \cite{Fernique-Ollinger-2010,Goodman-Strauss-1998,Mozes-1989}).
It has moreover recently been proven in \cite{Fernique-Sablik-2016} that a cut and project tiling admits local rules with decorated tiles if and only if it can be defined by {\em computable} quantities.
This complete characterization goes much further than previous results ({\em e.g.}, \cite{Le-Piunikhin-Sadov-1993,Le-1992b,Le-1992c,Le-1997,Le-1995,Le-Piunikhin-Sadov-1992,Socolar-1989}) by using decorations to simulate Turing machines.
But it can hardly help to model real quasicrystals because of the huge number of different decorated tiles that it needs.\\

If one does not allow tiles to be decorated, then the situation becomes more realistic but dramatically changes.
Algebraicity indeed comes into play instead of computability.
This problem has been widely studied (see, {\em e.g.}, \cite{Beenker-1982,deBruijn-1981,Burkov-1988,Katz-1988,Katz-1995,Kleman-Pavlovitch-1987,Le-1992c,Le-1997,Levitov-1988,Socolar-1990}), but there is yet no complete characterization.
We here provide the first such characterization in the case of so-called {\em octagonal tilings}.\\

Let us here sketch the main definitions leading up to our theorem (more details are given in Section~\ref{sec:settings}).
An {\em octagonal tiling} is a covering of the plane by rhombi whose edges have unit length and can take only four different directions, with the intersection of two rhombi beeing either empty, or a point, or a whole edge.
By interpretating these four edges as the projection of the standard basis of $\mathbb{R}^4$, any octagonal tiling can be seen as a square tiled surface in $\mathbb{R}^4$, called its {\em lift}. 
It is then said to be {\em planar} if this lift lies in the neighborhood $E+[0,t]^4$ of a $2$-dimensional affine plane $E\subset\mathbb{R}^4$, called the {\em slope} of the tiling.
Unless otherwise specified, ``plane'' shall here mean ``$2$-dimensional affine plane of $\mathbb{R}^4$''.\\

On the one hand, a plane $E$ is determined by its {\em subperiods} if any other plane having the same subperiods is parallel to $E$, where a subperiod of $E$ corresponds to a direction in $E$ with at least three rational entries\footnote{We shall formally define it as a linear integer relation on three Grassmann coordinates of $E$, see Definition~\ref{def:subperiod}.}.
In other words, $E$ is determined by some algebraic constraints.\\

On the other hand, a plane $E$ is said to admit {\em weak local rules} if there is $r\geq 0$ such that, whenever the patterns of size $r$ of an octagonal tiling form a subset of the patterns of a planar octagonal tiling with a lift in $E+[0,1]^4$, then this tiling is planar with slope $E$.
In other words, $E$ is determined by some geometric constraints.\\

\noindent Our main result connects these algebraic and geometric constraints:

\begin{theorem}\label{th:main}
A plane admitting weak local rules is determined by its subperiods.
\end{theorem}

This characterization is actually up to {\em algebraic conjugacy} in the sense that such a plane $E$ turns out to be always generated by vectors with entries in some quadratic number field $\mathbb{Q}(\sqrt{d})$ (see Cor.~\ref{cor:quadratic}, below) and the plane $E'$ obtained by changing $\sqrt{d}$ into $-\sqrt{d}$ everywhere also has the same subperiods (but octagonal tilings with a lift in $E'+[0,t]$ may not exist).
The converse implications is the main theorem of \cite{Bedaride-Fernique-2015}, so that we get a full characterization:

\begin{corollary}\label{cor:characterization}
A plane admits weak local rules if and only if it is determined by its subperiods.
\end{corollary}

This is moreover an {\em effective} characterization.
We indeed show how to associate with any given slope a system of polynomial equations which is zero-dimensional if and only if this slope is characterized by its subperiods.
The zero-dimensionality of such a system can then be checked by computer.
We will also easily obtain as a corollary the following result:

\begin{corollary}\label{cor:quadratic}
  If a plane has weak local rules, then it is generated by vectors with entries in some quadratic number field $\mathbb{Q}(\sqrt{d})$.
\end{corollary}

This answers a conjecture of Thang Le in the 90s.
He showed in \cite{Le-1997} that if the slope of a planar tiling (planar octagonal tilings are a particular case) has weak local rules, then it is generated by vectors with entries in a common algebraic field.
He conjectured that it is a quadratic field for $2$-planes of $\mathbb{R}^4$.\\

The maximal algebraic degree is however still unknown in general.
One can show that it would be $\lfloor n/d\rfloor$ if Theorem~\ref{th:main} extends to $d$-dimensional affine planes of $\mathbb{R}^n$.
At least, there is no counter-example to our knowledge.
For example, the slope of Penrose tilings is a $2$-dimensional affine plane of $\mathbb{R}^5$ based on the golden ratio which has degree $2=\lfloor 5/2\rfloor$.
More generally, the slope of an {\em $2p$-fold tiling} ($p\geq 3$) is a $2$-dimensional affine plane of $\mathbb{R}^p$ based on an algebraic number of degree $\varphi(p)/2\leq \lfloor p/4\rfloor$, where $\varphi$ is the Euler's totient function (the Penrose case corresponds to $p=5$).
Let us also mention the {\em icosahedral tiling}, whose slope is a $3$-dimensional affine space of $\mathbb{R}^6$ based, again, on the golden ratio, of degree $2=\lfloor 6/3\rfloor$.\\

The paper is organized as follows.
Section~\ref{sec:settings} introduces the settings, providing the necessary formal definitions, in particular weak local rules and subperiods.
Section~\ref{sec:less} proves that a plane with less than three types of subperiods cannot have weak local rules.
The idea is to construct a non-planar tiling which has the same patterns of a given size as the original planar tiling.
This relies on the precise study of what happens when the slope of a planar tiling is slightly shifted (Proposition~\ref{prop:shift_flips}).
Section~\ref{sec:more} proves that if a plane has weak local rules, hence three types of subperiods, then it has necessarily a fourth subperiod (Lemma~\ref{lem:fourth_subperiod}) which, together with the three first subperiods, characterize it.
This yields the main theorem.
The proof relies on a case-study which uses the notion of {\em coincidence} (Definition~\ref{def:coincidence}) to express in algebraic terms the constraints on patterns enforced by weak local rules.


\section{Settings}
\label{sec:settings}

Let $\vec{v}_1,\ldots,\vec{v}_4$ be pairwise non-colinear vectors of $\mathbb{R}^2$ and define the {\bf proto-tiles}
$$
T_{ij}=\{\lambda\vec{v}_i+\mu\vec{v}_j~|~0\leq\lambda,\mu\leq 1\},
$$
for $1\leq i<j\leq 4$.
A {\bf tile} is a translated proto-tile.
An {\bf octagonal tiling} is a edge-to-edge tiling by these tiles, that is, a covering of the Euclidean plane such that two tiles can intersect only in a vertex or along a whole edge.\\

The {\bf lift} of an octagonal tiling is a $2$-dim. surface of $\mathbb{R}^4$ defined as follows: an arbitrary vertex of the tiling is first mapped onto an arbitary point of $\mathbb{Z}^4$, then each tile $T_{ij}$ is mapped onto the unit face generated by $\vec{e}_i$ and $\vec{e}_j$, where $\vec{e}_1,\ldots,\vec{e}_4$ denote the standard basis of $\mathbb{R}^4$, so that two tiles adjacent along $\vec{v}_i$ are mapped onto faces adjacent along $\vec{e}_i$.\\

Among octagonal tilings, we distinguish {\bf planar octagonal tilings}: they are those with a lift which lies inside a tube $E+[0,t]^4$, where $E$ is a (two-dimensional) affine plane of $\mathbb{R}^4$ called the {\bf slope} of the tiling, and $t\geq 1$ is a real number called the {\bf thickness} of the tiling (both $E$ and $t$ are uniquely defined).\\


A plane is {\bf irrational} if it does not contain any line generated by a vector with only rational entries.
By extension, a planar tiling is said to be irrational if its slope is irrational.
An irrational tiling is {\bf aperiodic} or {\bf non-periodic}, {\em i.e.}, no (non-trivial) translation maps it onto itself.
It can actually be ``more or less irrational'' because they may exist rational dependencies between the {\bf Grassmann coordinates} of its slope.
Recall (see {\em e.g.}, \cite{Hodge-Pedoe-1994}, chap.~7, for a general introduction) that the Grassmann coordinates of a plane $E$ generated by two vectors $(u_1,u_2,u_3,u_4)$ and $(v_1,v_2,v_3,v_4)$ are the six real numbers defined up to a common multiplicative factor by
$$
G_{i,j}=u_iv_j-u_jv_i,
$$
for $1\leq i<j\leq 4$.
They always satisfy the so-called {\bf Plücker relation}:
$$
G_{12}G_{34}=G_{13}G_{24}-G_{14}G_{23}.
$$
A plane is said to be {\bf nondegenerate} if its Grassmann coordinates are all non zero.
By extension a planar tiling is said to be nondegenerate if its slope is nondegenerate: this means that each of the six proto-tiles appears in the tiling.
We will implicitly consider only such planes or tilings in this paper.\\

We used Grassmann coordinates in \cite{Bedaride-Fernique-2015} to rephrase the geometric {\em SI-condition} of \cite{Levitov-1988} in more algebraic terms via the notion of {\em subperiod}:

\begin{definition}[subperiod]\label{def:subperiod}
  A {\em type $k$ subperiod} of a plane $E$ is a linear rational equation on its three Grassmann coordinates which have no index $k$.
\end{definition}

One can show (Prop. 5 of \cite{Bedaride-Fernique-2015}) that a plane $E$ has a subperiod $pG_{23}-qG_{13}+rG_{12}=0$ of type $4$ if and only if there is $x\in\mathbb{R}$ such that $E$ contains a line directed by $(p,q,r,x)$ (this is how subperiods were defined in the introduction).\\
  
Consider, for example, the celebrated Ammann-Beenker tilings.
One of them is depicted on Fig.~\ref{fig:ammann_beenker_tiling}.
They are the planar octagonal tilings of thickness one with a slope parallel to the plane generated by
$$
(\sqrt{2},1,0,-1)
\qquad\textrm{and}\qquad
(0,1,\sqrt{2},1).
$$
This plane has Grassmann coordinates $(1,\sqrt{2},1,1,\sqrt{2},1)$ (by lexocographic order).
It is irrational but has four subperiods (ordered by increasing type):
$$
G_{23}=G_{34},
\qquad
G_{14}=G_{34},
\qquad
G_{12}=G_{14},
\qquad
G_{12}=G_{23}.
$$
However, one checks that these equations and the Plücker relation do not characterize the Ammann-Beenker slope but the one-parameter family of planes with Grassmann coordinates $(1,t,1,1,2/t,1)$ (see \cite{Bedaride-Fernique-2013}).
Hence, according to Theorem~\ref{th:main}, Ammann-Beenker tilings do not have weak local rules.
This particular case was already (differently) proven by Burkov in \cite{Burkov-1988} (see also \cite{Bedaride-Fernique-2015b}).\\

\begin{figure}[hbtp]
  \includegraphics[width=\textwidth]{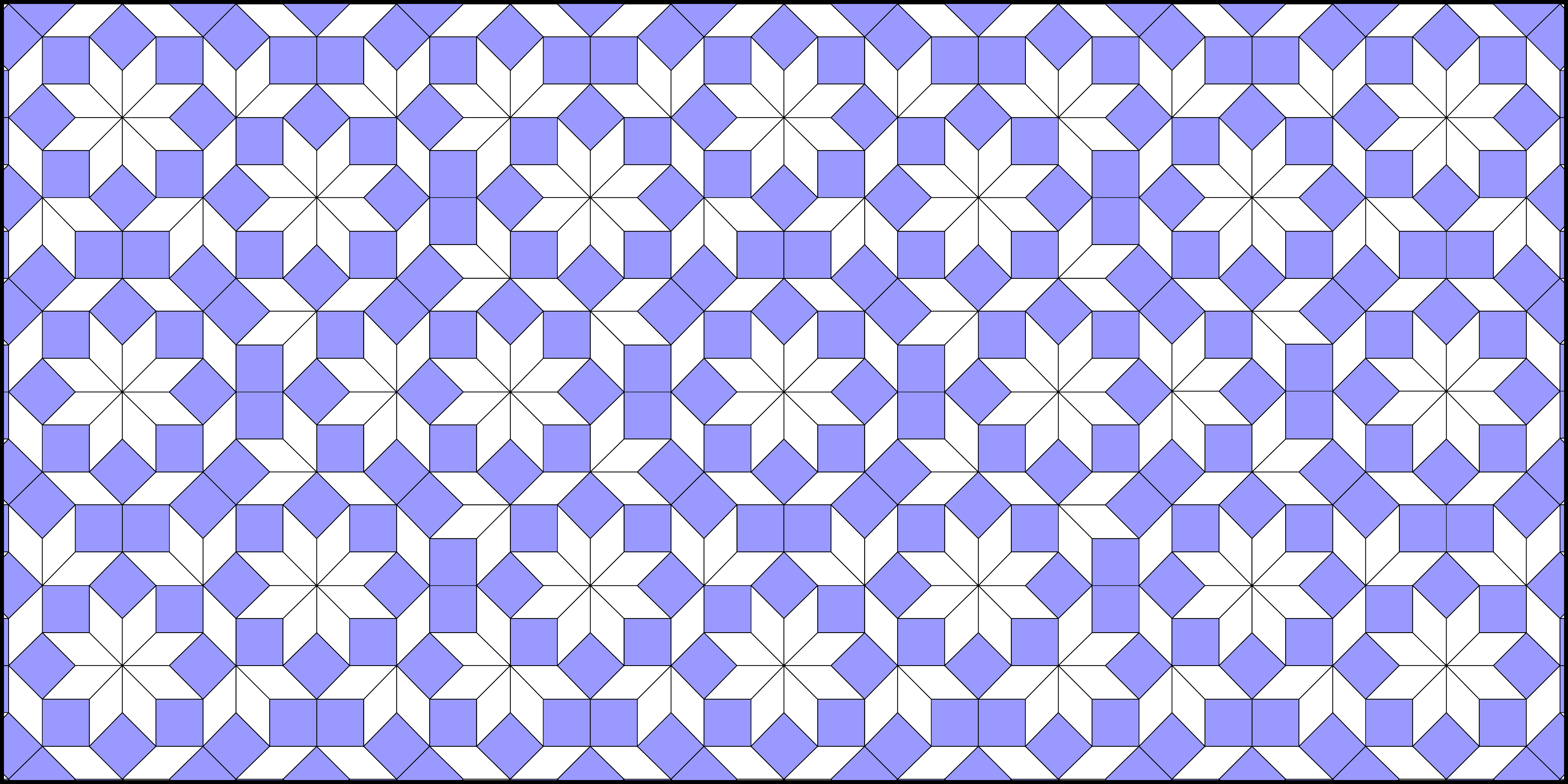}
  \caption{A celebrated octagonal tiling: the Ammann-Beenker tiling.}
  \label{fig:ammann_beenker_tiling}
\end{figure}


It is also worth noticing (although we shall not use this in this paper) that planar tilings provide a very natural interpretation for Grassmann coordinates.
The frequency of the proto-tile $T_{ij}$ in a planar octagonal tiling of slope $E$, which exists, indeed turns out to be proportional to $|G_{ij}|$.
In other words, one can ``read'' on a planar tiling the Grassmann coordinates of its slope.
For example, any Ammann-Beenker tiling contains $\sqrt{2}$ rhombi for one square (and both cover half of the plane since a square is $\sqrt{2}$ times larger than a rhombus, see Fig.~\ref{fig:ammann_beenker_tiling}).\\


A {\bf pattern} of an octagonal tiling is a finite subset of its tiles.
Among patterns, we distinguish {\bf $r$-maps}.
A $r$-map is a pattern whose tiles are exactly those intersecting a closed ball of diameter $r$ drawn on the tilings.
The set of $r$-maps of a tiling form its {\bf $r$-atlas}.
The main question we are here interested in is: when does the $r$-atlas of a tiling characterize it?
Formally, we follow \cite{Levitov-1988}:

\begin{definition}[weak local rules]\label{def:local_rules}
  A plane $E$ has {\em weak local rules} of diameter $r$ and thickness $t$ if any octagonal tiling whose $r$-maps are also $r$-maps of a planar tiling with slope $E$ and thickness $1$ is itself planar with slope $E$ and thickness $t$.
\end{definition}

In other words, a finite number of finite prescribed patterns (the $r$-atlas) suffices to enforce a tiling to have the slope $E$.
By extension, one says that a planar tiling admits weak local rules if so does its slope.
The parameter $t\geq 1$, allows some bounded fluctuations around $E$.
{\bf Strong local rules} corresponds to $t=1$.
This distinction between strong and weak local rules actually play a significant role.
For example, the so-called $7$-fold tilings, based on cubic irrationalities, have weak local rules \cite{Socolar-1990} but no strong local rules \cite{Levitov-1988}.
Theorem~\ref{th:main} {\em a fortiori} holds for strong local rules, but the result proven in \cite{Bedaride-Fernique-2015} allows to state corollary~\ref{cor:characterization} only in terms of weak local rules.\\

Let us now briefly recall the notion of {\bf window} and some of its properties (a complete presentation can be found in \cite{Baake-Grimm-2013}, Chapter 7).
The window of a planar octagonal tiling with slope $E$ and thickness $1$ is the octagon obtained by orthogonally projecting $[0,1]^4$ onto the orthogonal plane $E^\bot$.
One can then associate with any pattern $\mathcal{P}$ of this tiling a polygonal region $R(\mathcal{P})$ of its window, such that $\mathcal{P}$ appears in position $\vec{x}$ if and only if the projection of $\vec{x}$ in the window falls in $R(\mathcal{P})$.
This is for example used in \cite{Julien-2010} to compute the {\em complexity} of tilings, that is, the number of its patterns with a given size.
The following proposition, which is a particular case of Prop.~3.5 in \cite{Le-1997} or Prop.~1 in \cite{Levitov-1988}, can then be deduced from the density of the projection of $\mathbb{Z}^4$ in the window:

\begin{proposition}\label{prop:LI_classes}
  Two planar irrational octagonal tilings with parallel slope and thickness $1$ have the same patterns.
\end{proposition}

For example, the Ammann-Beenker tilings all have the same patterns.
This may explain why one often speaks about the Ammann-Beenker tiling in the singular form (as in the caption of Fig.~\ref{fig:ammann_beenker_tiling}) although there is uncountably many of them.
We will here use the window to look how patterns appear when the slope $E$ is modified.
We rely on the following notion, introduced in \cite{Bedaride-Fernique-2015b}:

\begin{definition}[coincidence]\label{def:coincidence}
  A {\em coincidence} of a plane $E\subset\mathbb{R}^4$ is a point of the window of $E$ which lies on the orthogonal projection of (at least) three unit open line segments with endpoints in $\mathbb{Z}^4$.
\end{definition}

Coincidences are exactly the points where new patterns can appear when the slope is modified.
Indeed, the boundary of the region $R(\mathcal{P})$ associated with a pattern $\mathcal{P}$ turns out to be delimited by the projection of line segments of $\mathbb{Z}^4$.
Hence, in order to create a new pattern, the slope must be modified so that the projection of $k\geq 3$ line segments of $\mathbb{Z}^4$ that formed a coincidence now form a nonempty polygonal region (see Fig.~\ref{fig:coincidence}).
One can moreover show that a plane which is not determined (among the planes) by a finite set of coincidences can, for any $r$, be modified without creating a region associated with a pattern of size $r$.
In other words (see \cite{Bedaride-Fernique-2015b}, Prop.~3):

\begin{proposition}\label{prop:coincidence}
  If a plane has weak local rules, then it is determined by finitely many coincidences.
\end{proposition}

\begin{figure}[hbtp]
  \includegraphics[width=\textwidth]{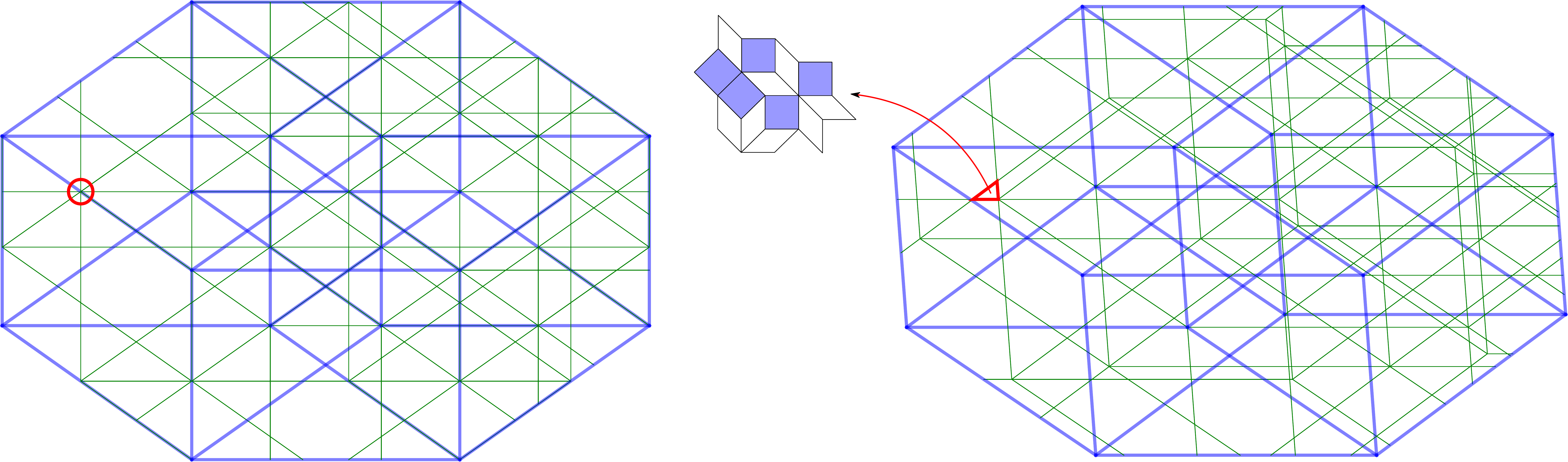}
  \caption{
    The window of an Ammann-Beenker tiling, divided into some regions (each of which corresponds to a pattern), with a circled coincidence (left).
    The slope is slightly changed so that the circled coincidence breaks and a new region appears (right).
    This corresponds to a new pattern which does not appear in an Ammann-Beenker tiling (compare with Fig.~\ref{fig:ammann_beenker_tiling}).
  }
  \label{fig:coincidence}
\end{figure}

\noindent Coincidences can also be expressed in terms of Grassmann coordinates:

\begin{proposition}\label{prop:coincidence2}
  A coincidence of a plane corresponds to an equation on its Grassmann coordinates of the form (up to a permutation of the indices)
  $$
  aG_{14}G_{23}-bG_{13}G_{24}+cG_{12}G_{34}+dG_{24}G_{34}+eG_{14}G_{34}+fG_{14}G_{24}=0,
  $$
  where $a$, $b$, $c$, $d$, $e$ and $f$ are integers.
\end{proposition}

\begin{proof}
  Consider a coincidence.
  It is the intersection of the projection of three unit open line segments with endpoints in $\mathbb{Z}^4$.
  Consider, on each of these segments, the point which projects onto the coincidence.  
  This yields three points which can be written (up to a permutation of the indices)
  $$
  (x,a,b,c), \qquad (d,y,e,f),\qquad (g,h,z,i),
  $$
  where all the entries are integers except $x$, $y$ and $z$.
  Let $(u_1,u_2,u_3,u_4)$ and $(v_1,v_2,v_3,v_4)$ be a basis of the plane.
  There are $\lambda_1$, $\mu_1$, $\lambda_2$ and $\mu_2$ such that
  $$
  \left(\begin{array}{c}
    x-d\\
    a-y\\
    b-e\\
    c-f
  \end{array}\right)
  =
  \lambda_1
  \left(\begin{array}{c}
    u_1\\
    u_2\\
    u_3\\
    u_4
  \end{array}\right)
  +\mu_1
  \left(\begin{array}{c}
    v_1\\
    v_2\\
    v_3\\
    v_4
  \end{array}\right)
  $$
  and
  $$
  \left(\begin{array}{c}
x-g\\
a-h\\
b-z\\
c-i
\end{array}\right)
=
\lambda_2
\left(\begin{array}{c}
u_1\\
u_2\\
u_3\\
u_4
\end{array}\right)
+\mu_2
\left(\begin{array}{c}
v_1\\
v_2\\
v_3\\
v_4
\end{array}\right).
$$
The third and fourth entries of the first equation yield
$$
\lambda_1=\frac{(b-e)v_4-(c-f)v_3}{G_{34}} \qquad \textrm{and} \qquad \mu_1=\frac{-(b-e)u_4+(c-f)u_3}{G_{34}}.
$$
The second and fourth entries of the second equation yield
$$
\lambda_2=\frac{(a-h)v_4-(c-i)v_2}{G_{24}} \qquad \textrm{and} \qquad \mu_2=\frac{-(a-h)u_4+(c-i)u_2}{G_{24}}.
$$
The first entry of these equations yields two expressions for $x$:
$$
x=d+\lambda_1u_1+\mu_1v_1=g+\lambda_2u_1+\mu_2v_1.
$$
By replacing $\lambda_1$, $\mu_1$, $\lambda_2$ and $\mu_2$ by their expressions and by grouping the terms in order to make appearing Grassmann entries, the previous inequality becomes
$$
d+\frac{(b-e)G_{14}-(c-f)G_{13}}{G_{34}}
=
g+\frac{(a-h)G_{14}-(c-i)G_{12}}{G_{24}}.
$$
By multiplying by $G_{24}G_{34}$ we get the claimed equation (with $a=0$).
\end{proof}

In the proof of the above proposition, we got $a=0$.
But the Plücker relation allows to replace any element in $\{G_{12}G_{34},G_{13}G_{24},G_{14}G_{23}\}$ by an integer linear combination of the two other ones.
We could thus equally have $b=0$ or $c=0$.
We have nevertheless chosen to write the three terms in order to emphasize the symmetry of this equation.
Indeed, one checks that any permutation of the indices $\{1,2,3\}$ yields the same permutation of the coefficients $\{a,b,c\}$ (with a change of sign) and $\{d,e,f\}$, hence do not modify the form of the equation.
The index $4$ plays here a special role because this coincidence corresponds to the projection of unit segments directed by $\vec{e}_1$, $\vec{e}_2$ and $\vec{e}_3$.


\section{At most two types of subperiods}
\label{sec:less}

In this section, we show that a plane $E$ with at most two types of subperiods cannot admit weak local rules.\\

Let us first informally sketch the proof.
The idea is to look at how spread the integer points which enter or exit the tube $E+[0,1]^4$ when $E$ is shifted.
On a planar tiling of slope $E$ and thickness $1$, this corresponds to a local rearrangement of tiles called {\em flip} (physicists speak about {\em phason flips}).
First, we shall show that, for any $r$, these flips can be made sparse enough to draw on $E$ a curve which stays at distance at least $r$ from any of them.
Then, we shall shift $E$ only on one side of this curve in order to create in the tiling a ``step'' that cannot be detected by patterns of diameter $r$.
Last, by repeating this, we shall build a sort of ``staircase'' which has the same $r$-atlas as the original tiling but is not planar (hence contradicting the existence of weak local rules).\\

\noindent Formally, let us associate with any shift vector $\vec{s}\in\mathbb{R}^4$ the set
$$
E(\vec{s}):=\{x\in\mathbb{Z}^4,~x\in (E+\vec{s}+[0,1]^4)\backslash(E+[0,1]^4)\}.
$$
The following proposition is illustrated by Figures~\ref{fig:four_subperiods} and \ref{fig:two_subperiods}.

\begin{proposition}\label{prop:shift_flips}
  Let $E$ be a plane of $\mathbb{R}^4$ and $r\geq 0$.
  Then, for $\vec{s}$ small enough, one can writes $E(\vec{s})=E_1\cup E_2\cup E_3\cup E_4$, where $E_i$ is empty if $\vec{s}\in\mathbb{R}\vec{e}_i$, or can be described according to the number of subperiods of type $i$ of $E$ otherwise:
  \begin{description}
  \item[0 subperiod:] any two points in $E_i$ are at distance at least $r$ from each other;
  \item[1 subperiod:] there is a set of parallel lines of $E$ at distance at least $r$ from each other, whose direction depends only on the subperiod, and such that the points in $E_i$ are within distance $1$ from these lines;
  \item[2 subperiods:] the points of $E_i$ are within distance $1$ from a lattice.
  \end{description}
\end{proposition}

\begin{proof}
  For $i=1,\ldots,4$, define
  $$
  E_i:=\{x\in\mathbb{Z}^4,~x\in (E+\vec{s}+[0,1]^4)\backslash(E+[0,1]^4+\mathbb{R}\vec{e}_i)\}.
  $$
  This set is empty when $\vec{s}\in\mathbb{R}\vec{e}_i$, and one has:
  $$
  \cup_i E_i=\{x\in\mathbb{Z}^4,~x\in (E+\vec{s}+[0,1]^4)\backslash\cap_i(E+[0,1]^4+\mathbb{R}\vec{e}_i)\}=E(\vec{s}).
  $$
  Assume, {\em w.l.o.g.}, that $i=1$ and consider $E_1$.\\
  
  Actually, let us first consider $\pi_1(E_1)$ where $\pi_1$ denote the projection which removes the first entry:
  $$
  \pi_1(E_1)=\{x\in\mathbb{Z}^3,~x\in (\pi_1(E)+\pi_1(\vec{s})+[0,1]^3)\backslash(\pi_1(E)+[0,1]^3)\}.
  $$
  This is the set of the integer points of the Euclidean space lying between two planes parallel to $\pi_1(E)$, with one being the image of the other by a translation by $\pi_1(\vec{s})$.
  If $E$ is generated by $(u_1,u_2,u_3,u_4)$ and $(v_1,v_2,v_3,v_4)$, then $\pi_1(E)$ is generated by $(u_2,u_3,u_4)$ and $(v_2,v_3,v_4)$.
  The cross product of these two vectors yields a normal vector for $\pi_1(E)$.
  One computes $(G_{23},-G_{24},G_{34})$, where the $G_{ij}$'s are the Grassmann coordinates of $E$.
  One can thus rewrite:
  $$
  \pi_1(E_1)=\{(a,b,c)\in\mathbb{Z}^3,~z\leq aG_{23}-bG_{24}+cG_{34}\leq z+f(\vec{s})\},
  $$
  where $z$ depends only on how the unit cube $[0,1]^3$ projects onto the line directed by $(G_{23},-G_{24},G_{34})$, while $f(\vec{s})$ is the dot product of $(G_{23},-G_{24},G_{34})$ and $\pi_1(\vec{s})$.
  In particular, $f(\vec{s})$ tends towards $0$ when $\vec{s}$ tends towards $\vec{0}$.
  Now, assume that for any $\vec{s}$, $\pi_1(E_1)$ contains two integer points $x(\vec{s})$ and $y(\vec{s})$ at distance at most $r$ from each other.
  The non-zero integer vector $d(\vec{s}):=x(\vec{s})-y(\vec{s})$ takes finitely many values.
  So does also its dot product with $(G_{23},-G_{24},G_{34})$, which is in the interval $[-f(\vec{s}),f(\vec{s})]$.
  This latter is thus necessarily equal to zero when this interval is small enough, that is, for a small enough $\vec{s}$.
  For such a $\vec{s}$, we have
  $$
  d_1(\vec{s})G_{23}-d_2(\vec{s})G_{24}+d_3(\vec{s})G_{34}=0,
  $$
  where $d_i(\vec{s})$ denotes the $i$-th entry of $d(\vec{s})$.
  Since $d(\vec{s})$ is a non-zero integer vector, this is exactly the equation of a subperiod of type $1$.
  In other words, for a small enough $\vec{s}$, any two points in $\pi_1(E_1)$ at distance at most $r$ are aligned along a direction determined by a subperiod of type $1$ of $E$.
  There is thus no such point if $E$ does not have a subperiod of type $1$.
  If $E$ has exactly one subperiod, then the points must be on parallel lines, with the distance between two lines being less than $r$ (otherwise it would yields a second subperiod).
  If $E$ has two subperiods, then the points are on a lattice.
  We thus have the wanted description\ldots but for $\pi_1(E_1)$!\\
  
  Let us come back to $E_1$.
  Let $x$ and $y$ in $E_1$.
  The points $\pi_1(x)$ and $\pi_1(y)$ are at distance at most $f(\vec{s})$ from $\pi_1(E_1)$.
  There is thus a vector $\vec{p}\in E$ and $\vec{q}$ of length at most $f(\vec{s})$ such that
  $$
  \pi_1(x-y)=\vec{p}+\vec{q}.
  $$
  By definition of $\pi_1$, there is then $k\in\mathbb{Z}$ such that
  $$
  x-y=\vec{p}+\vec{q}+k\vec{e}_1.
  $$
  Now, let $\pi'$ denote the orthogonal projection onto $E^\bot$.
  One has
  $$
  \underbrace{\pi'(x-y)}_{\in\pi'(E_1(\vec{s}))}=\underbrace{\pi'(\vec{p})}_{=0}+\underbrace{\pi'(\vec{q})}_{||.||\leq f(\vec{s})}+k\pi'(\vec{e}_1).
  $$
  The set $\pi'(E_1(\vec{s})$ is bounded and its closure converges towards a unit segment directed by $\vec{e}_1$ when $\vec{s}$ tends towards $\vec{0}$.
  For $\vec{s}$ small enough, this yields $k=0$ or $k=1$.
  The points in $E_1$ thus spread as the ones in $\pi_1(E_1)$ do, up to a small local correction by $\vec{e}_1$.
  This shows the claimed result.
\end{proof}

\begin{figure}[hbtp]
  \centering
  \includegraphics[width=0.48\textwidth]{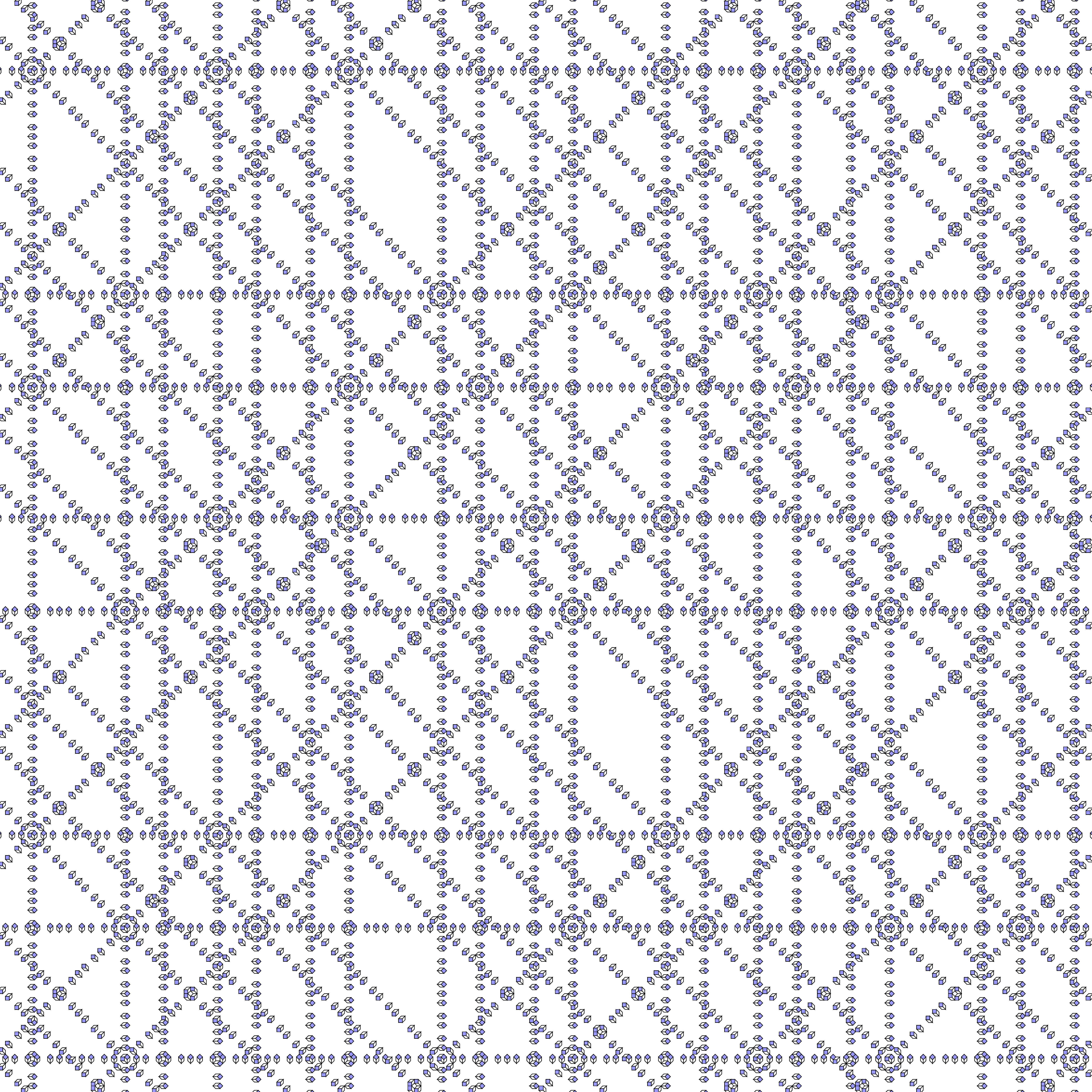}
  \hfill
  \includegraphics[width=0.48\textwidth]{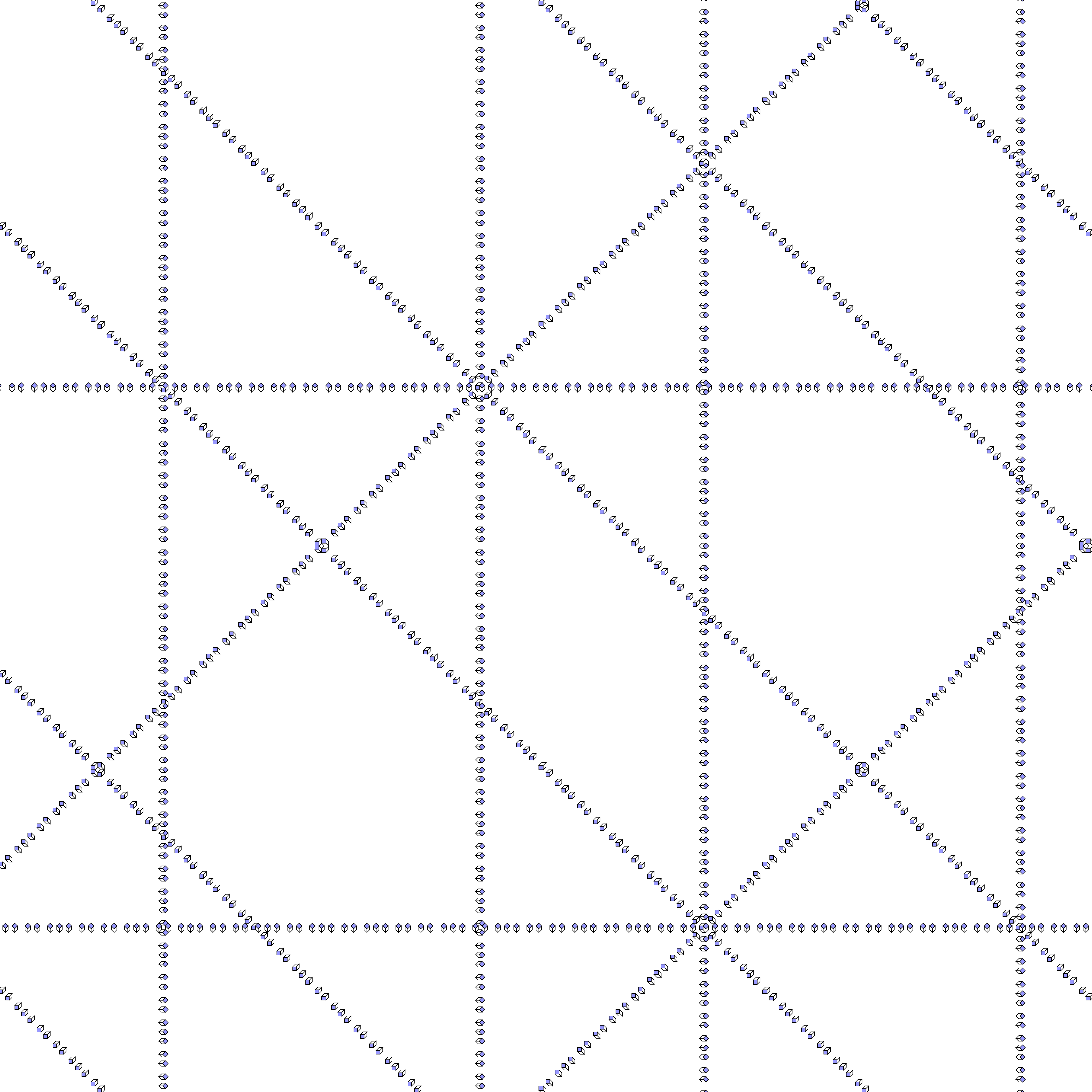}
  \caption{
    When an Ammann-Beenker tiling is shifted, it creates lines of flips whose directions are determined by its four subperiods (left).
    A smaller shift yields a similar picture, but the lines become sparser (right).
  }
  \label{fig:four_subperiods}
\end{figure}

\begin{figure}[hbtp]
  \centering
  \includegraphics[width=0.48\textwidth]{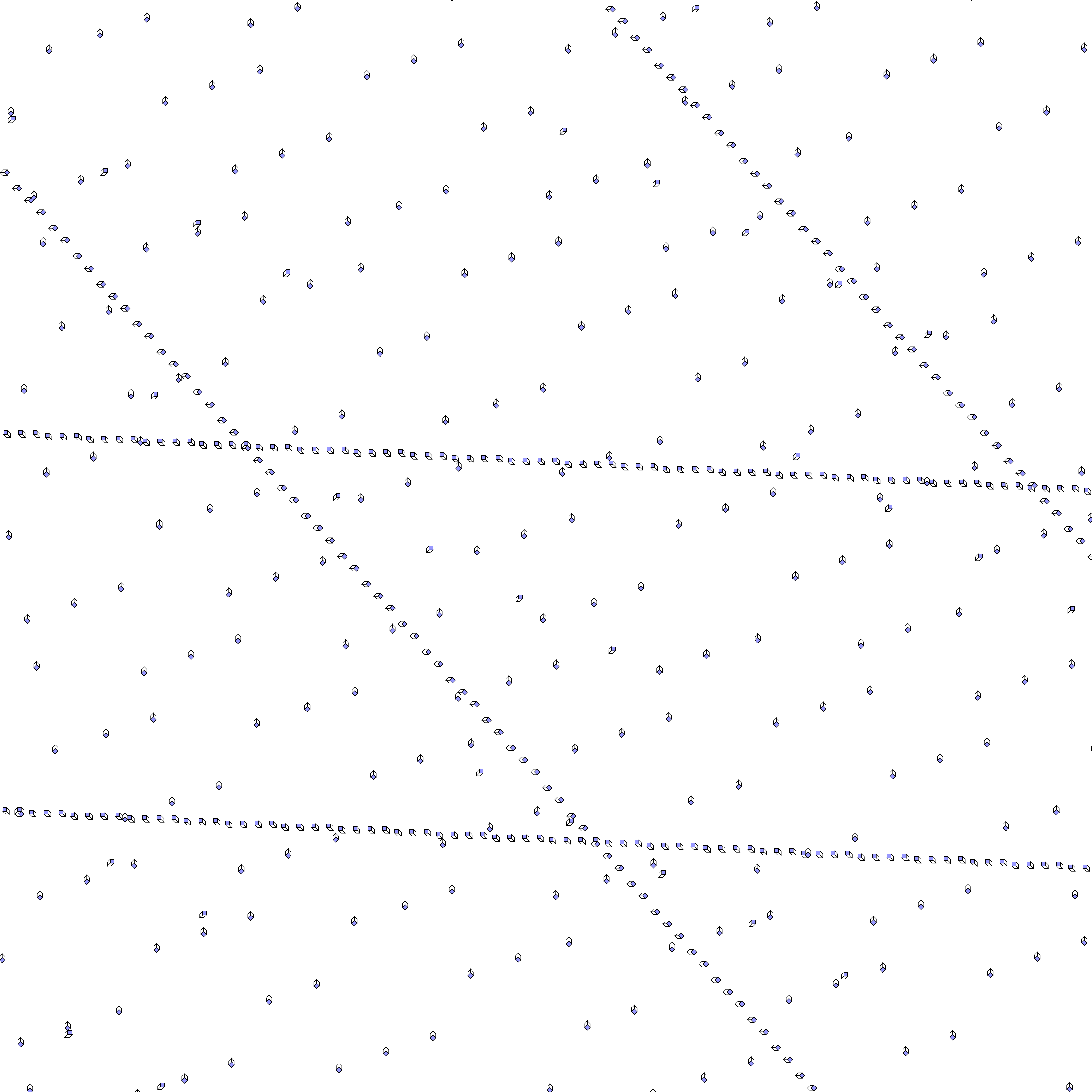}
  \hfill
  \includegraphics[width=0.48\textwidth]{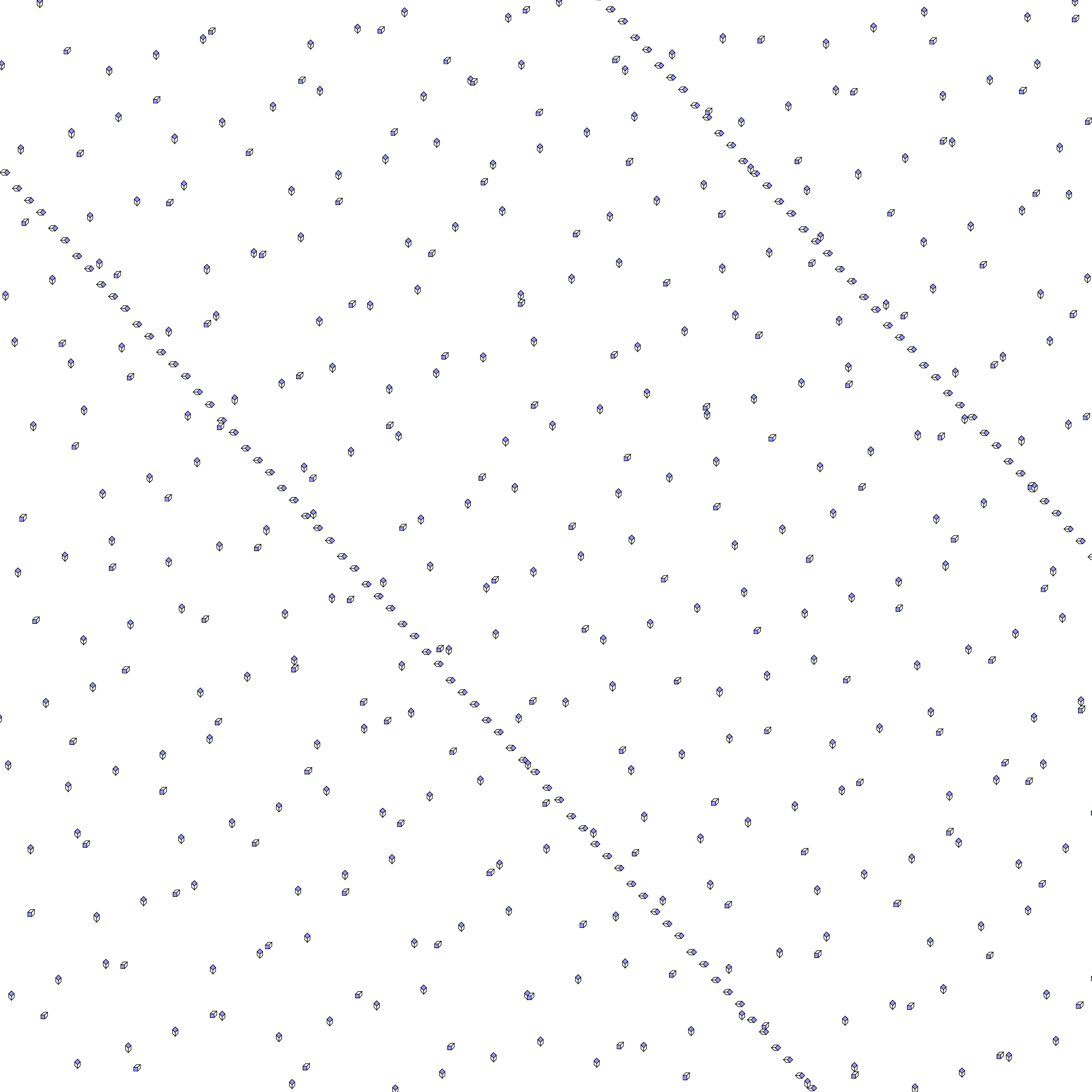}
  \caption{
    The irrational slope $(1,\sqrt{2},\sqrt{3},2\sqrt{2},3\sqrt{3},\sqrt{6})$ has one subperiod of type $3$ and one of type $4$.
    A generic shift thus creates two sets of sparse lines and two sets of sparse flips (left).
    However, a shift along $\vec{e}_4$ ``neutralizes'' the lines of flips directed by the subperiod of type $4$ (right).
    }
  \label{fig:two_subperiods}
\end{figure}

\noindent We shall need the following elementary lemma:

\begin{lemma}\label{lem:four_subperiods_two_types}
A plane with two subperiods of type $i$ and two of type $j$ is rational.
\end{lemma}

\begin{proof}
  Assume, {\em w.l.o.g.}, that $i=1$ and $j=2$.
  The two subperiods of type $1$ yield two independent rational equations on $G_{23}$, $G_{24}$ and $G_{34}$.
  Hence $G_{23}$, $G_{24}$ are in $\mathbb{Q}(G_{34})$.
  The two subperiods of type $2$ yield two independent rational equations on $G_{13}$, $G_{14}$ and $G_{34}$.
  Hence $G_{13}$, $G_{14}$ are in $\mathbb{Q}(G_{34})$.
  The Plücker relations then yields that $G_{12}$ is also in $\mathbb{Q}(G_{34})$.
  The Grassmann coordinates are thus all in the same number field.
  This shows that such a plane is rational, actually {\em totally} rational ({\em i.e.}, it contains two independent rational lines).
\end{proof}

\begin{proposition}\label{prop:at_most_two_subperiods}
An irrational plane with at most two types of subperiods does not admit weak local rules.
\end{proposition}

\begin{proof}
  Let $E$ be an irrational plane with at most two types of subperiods.
  Let $\mathcal{T}$ be a planar tiling with slope $E$ and thickness $1$. 
  By Lemma~\ref{lem:four_subperiods_two_types}, $E$ has at most three subperiods, say two of type $4$ and one of type $3$.
  Fix $r\geq 0$.
  We shall construct step by step a ``staircase'' tiling $\mathcal{T}_\infty$ with the same $r$-atlas as $\mathcal{T}$ but which is non-planar.
  This will prove that $E$ has no weak local rules of diameter $r$.\\
  
  {\bf Step height.}
  Prop.~\ref{prop:shift_flips} yields a nonzero vector $\vec{s}\in\mathbb{R}\vec{e}_4$ such that $E(\vec{s})$ can be written $E_1\cup E_2\cup E_3$, where points in $E_1$ are at distance $1$ from parallel lines of $E$ at distance $3r$ from each other, while $E_2$ and $E_3$ are points at distance $3r$ from each other.
  This shift $\vec{s}$ will be the step height of our staircase.\\

  {\bf Step edge.}
  Consider two consecutive parallel lines of $E$ from which the points in $E_1$ are at distance at most $1$.
  Between them, there is a stripe of $E$ of width at least $2r$ which stays at distance at least $r/2$ from $E_1$.
  We claim that one can divide $E$ in two parts by drawing in this stripe a curve which stays at distance at least $r/2$ from any point in $E(\vec{s})$.
  In other words, this stripe cannot be blocked by the balls of diameter $r$ centered on the points in $E_2$ and $E_3$.
  Indeed, since this stripe has width $2r$, one needs at least three intersecting such balls to block it.
  Two of these balls must be centered on two points in the same set $E_2$ or $E_3$.
  These two points should be at distance at most $2r$ (because of the diameter of the balls), but it is impossible because two points in one of these sets are at distance at least $3r$ from each other.
  One can thus draw the wanted curve.
  It defines the edge of our step.
  Fig.~\ref{fig:step_edge} illustrates this.\\

  \begin{figure}[hbtp]
  \centering
  \includegraphics[width=\textwidth]{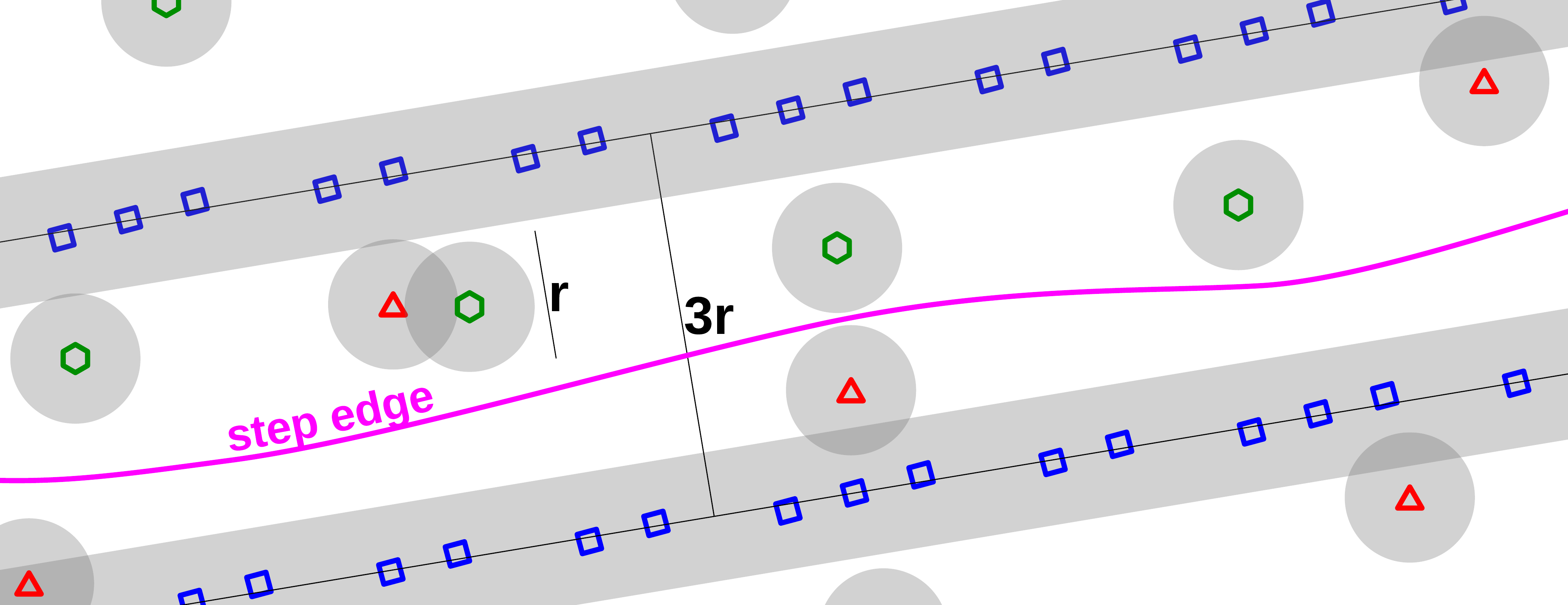}
  \caption{
    The points in $E_1$ are depicted by squares, while the points in $E_2$ and $E_3$ are respectively depicted by triangles and hexagons.
    The two lines are at distance at least $3r$ from each other, as well as any two triangles and any two hexagons.
    The step edge goes between the two lines and stay at distance at least $r/2$ from any point in $E(\vec{s})$, that is, outside the shaded region.
  }
  \label{fig:step_edge}
  \end{figure}

  {\bf First step.}
  Let us shift by $\vec{s}$ the part of $E$ on the right\footnote{We fix an orientation of the plane $E$ and keep the same orientation when we shift it.} of the previous curve, that is, we perform on $\mathcal{T}$ the flips which correspond to the points of $E(\vec{s})$ on the right of this curve.
  Let $\mathcal{T}_1$ denote the resulting tiling.
  We claim that $\mathcal{T}_1$ and $\mathcal{T}$ have the same $r$-atlas.
  On the one hand, every $r$-map of $\mathcal{T}$ can be found on the left side of the step (that is, in $\mathcal{T}$) because $\mathcal{T}$ is repetitive (any pattern which occurs once reoccurs at uniformly bounded distance from any point of the tiling).
  On the other hand, consider a $r$-map of $\mathcal{T}_1$.
  If it on the left side of the step, then it is in $\mathcal{T}$, thus in its $r$-atlas.
  If it is on the right side of the step, then it is in a tiling of slope $E+\vec{s}$, which has the same $r$-atlas by Prop.~\ref{prop:LI_classes}.
  It is thus also in the $r$-atlas of $\mathcal{T}$.
  If it crosses the curve defining the step, then since there is no point of $E(\vec{s})$ at distance less than $r/2$ from this curve, then this $r$-map can be seen as a pattern either of the tiling of slope $E$ or of the one of slope $E+\vec{s}$ (this is exactly what we cared about when defining the step).
  In both case it is in the $r$-atlas of $\mathcal{T}$.\\
 
  {\bf Next steps.}
  We proceed by induction.
  Let $\mathcal{T}_k$ be the tiling obtained after $k$ steps.
  It has the same $r$-atlas as $\mathcal{T}$.
  Its $i$-th step coincide with the tiling of thickness $1$ and slope $E+i\vec{s}$, either on a half plane for $i=0$ and $i=k$, or on a stripe otherwise.
  We obtain $\mathcal{T}_{k+1}$ by proceeding on the $k$-th step of $\mathcal{T}_k$ as we did on $\mathcal{T}$ in order to obtain $\mathcal{T}_1$ (we simply take care that the curve which defines the $k+1$-th step is far enough on the right of the one define the $k$-th step).
  Fig.~\ref{fig:staircase} illustrates this.\\
  
  \begin{figure}[hbtp]
    \centering
    \includegraphics[width=\textwidth]{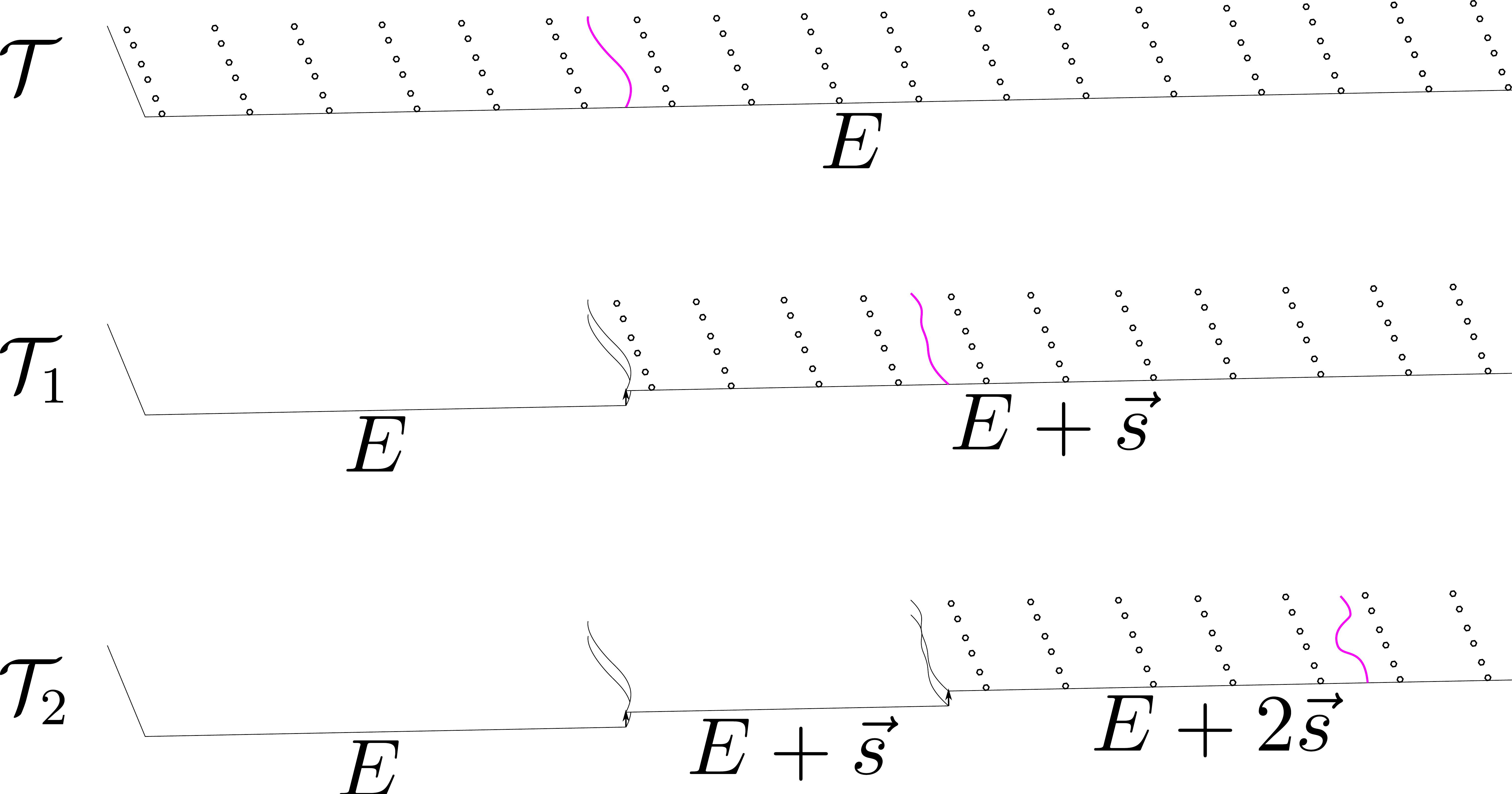}
    \caption{
      The creation of the three first steps of the staircase (from top to bottom).
      At each time, a curve (step edge) is drawn between two lines of points in $E_1$, and half of the tiling is shifted by $\vec{s}$.
      The changes from a steps to the next one are sparse enough to be undetectable by patterns of diameter $r$.
    }
    \label{fig:staircase}
  \end{figure}
  
  {\bf Staircase.}
  Let $\mathcal{T}_\infty$ be the tiling which coincides with $\mathcal{T}_k$ on its $k$ first steps.
  It has infinitely many steps: this is our staircase.
  It still has the same $r$-atlas as $\mathcal{T}$.
  It coincides with $\mathcal{T}$ on a half-plane, so must have slope $E$ if it is planar.
  But this cannot be the case because its $k$-th step is close to $E+k\vec{s}$, so cannot stay at bounded distance from $E$ for $k$ large enough.
  This proves that $E$ has no weak local rules of diameter $r$.
\end{proof}

\section{Three types of subperiods}
\label{sec:more}

In this section, we prove Theorem~\ref{th:main}.
The proof relies on Proposition~\ref{prop:at_most_two_subperiods} (previous section) and on two technical lemmas.
For the sake of clarity, let us first state these lemmas, then prove the theorem, and only after that prove the two lemmas.

\begin{lemma}\label{lem:thee_types_independence}
  If an irrational plane has three subperiods of different types, then these subperiods are independent.
\end{lemma}

\begin{lemma}\label{lem:fourth_subperiod}
  If an irrational plane is characterized by three subperiods of different types and a coincidence, then it is actually characterized solely by its subperiods.
\end{lemma}

\begin{proof}[Proof of Theorem~\ref{th:main}]
  Assume that $E$ is an irrational plane which admits weak local rules.
  According to Proposition~\ref{prop:coincidence}, it is determined by finitely many coincidences.
  We want to show that $E$ is actually determined by its subperiods.
  According to Proposition \ref{prop:at_most_two_subperiods}, we know that $E$ must have three subperiods of different type.
  Lemma~\ref{lem:thee_types_independence} ensures that these subperiods are independent.
  Since the space of two-dimensional planes in $\mathbb{R}^4$ has dimension $4$, these three subperiods form a one-dimensional system of equations.
  But since coincidences determine a zero-dimensional system, we can find a coincidence which, together with these three subperiods, form a zero-dimensional system, that is, characterizes the slope up to algebraic conjugacy.
  Lemma~\ref{lem:fourth_subperiod} then ensures that the plane is characterized, up to algebraic conjugacy, by its subperiods.
\end{proof}

In other terms, subperiods are as powerful as weak local rules.
Corollary~\ref{cor:quadratic} then just comes from the fact that the plane is characterized by linear rational equations (the subperiods) and a single quadratic rational equation: the Plücker relation.
Let us now prove the two technical lemmas.

\begin{proof}[Proof of Lemma~\ref{lem:thee_types_independence}]
  First, assume that two subperiods, say of type $1$ and $2$, are dependent:
  $$
  aG_{23}+bG_{24}+cG_{34}=0,
  $$
  $$
  dG_{13}+eG_{34}+fG_{14}=0,
  $$
  where the coefficients are rational.
  The dependence yields \mbox{$a=b=d=f=0$}, thus $c\neq 0$ and $e\neq 0$, whence $G_{34}=0$.
  This is forbidden because $E$ is nondegenerate.\\
  Now, assume that there are three dependent subperiods, say of type $1$, $2$ and $3$.
  The two first are written above, and the last one is
  $$
  gG_{12}+hG_{14}+iG_{24}=0,
  $$
  where the coefficients are rational.
  The dependence yields $a=d=g=0$, so that the equations yield that $G_{14}$, $G_{24}$ and $G_{34}$ are commensurate.
  But one checks that the vector $(G_{14},G_{24},G_{34},0)$ always belongs to $E$.
  This contradicts the irrationality of $E$.
  Hence three subperiods of different type are independent.
\end{proof}

\begin{proof}[Proof of Lemma~\ref{lem:fourth_subperiod}]
  We first prove that the Grassmann coordinates of such a plane can be chosen in the same quadratic field.
  This will yield a linear rational relation between any three Grassmann coordinates, in particular a subperiod of the fourth type.
  We then prove that this fourth subperiod is independent from the three first ones, so that all together they characterize the plane.\\

  \noindent {\em W.l.o.g.}, the coincidence equation can be written
  $$
  \circ G_{12}G_{34}+\circ G_{13}G_{24}+\circ G_{14}G_{24}+\circ G_{14}G_{34}+\circ G_{24}G_{34}=0,
  $$
  where the symbol $\circ$ stands for any rational number (this notation shall be used again in what follows).
  This equation being invariant (up to the coefficients) under any permutation of the indices $1$, $2$ and $3$ (by using the Plücker relation), there is two cases, depending whether there is subperiod of type $4$ (case A) or not (case B).
  In case A, one can assume (again, thanks to the invariance under permutation of the indices $1$, $2$ and $3$) that the two other subperiods have type $2$ and $3$.\\
  
  {\em Disclaimer}: it unfortunately seems hard to further use symmetry to shorten the following case study, although all the considered cases behave similarly\ldots

  \paragraph{Grassmann coordinates are quadratic: case A.}~\\
  The three subperiods are
\begin{eqnarray*}
pG_{23}&=&aG_{12}+bG_{13},\\
qG_{24}&=&cG_{12}+dG_{14},\\
rG_{34}&=&eG_{13}+fG_{14}.
\end{eqnarray*}
When $p$, $q$ or $r$ are non-zero, one can normalize them to one.
According to this, we make eight subcases, with three of which being eventually excluded.\\
{\bf Subcase A1}: $pqr\neq 0$.\\
Set $G_{12}=y$, $G_{13}=x$ and normalize to $G_{14}=1$.
The subperiods yield
$$
G_{23}=ay+bx,\qquad
G_{24}=cy+d,\qquad
G_{34}=ex+f.
$$
The Plücker relation becomes
$$
y(ex+f)=x(cy+d)-(ay+bx),
$$
that one simply writes (with again $\circ$ denoting a generic rational number)
$$
\circ xy+\circ x+\circ y=0.
$$
Since the plane is characterized, the system formed by the three subperiods, the Plücker relation and the coincidence equation is zero-dimensional.
The Plücker relation thus not reduces to $0=0$, that is, the coefficients of $xy$, $x$ and $y$ are not all equal to zero.
Moreover, since the plane is nondegenerate, {\em i.e.}, its Grassmann coordinates are non-zero, there is actually at most one of the coefficients of $xy$, $x$ and $y$ that can be equal to zero.
We can thus write
$$
x=\frac{\circ y}{\circ y+\circ}
$$
and rewrite the coincidence equation
$$
\circ xy+\circ x+\circ y+\circ=0.
$$
Using then the expression for $x$ obtained from the Plücker relation, and reducing to the same (non-zero) denominator, we get the equation
$$
\circ y^2+\circ y+\circ=0.
$$
This yields that $y$ is a quadratic number.
The expression for $x$ obtained from the Plücker relation then yields that $x$ belongs to the same quadratic number field.
The subperiods finally yields the same for all the other Grassmann coordinates.\\
{\bf Subcase A2}: $p=0$, $qr\neq 0$.\\
Set $G_{14}=x$, $G_{23}=y$ and normalize to $G_{12}=1$.
The subperiods yield
$$
G_{13}=\circ,\qquad
G_{24}=\circ x+\circ,\qquad
G_{34}=\circ x+\circ.
$$
The Plücker and coincidence relations respectively become
$$
\circ xy+\circ x+\circ=0
\qquad\textrm{and}\qquad
\circ x^2+\circ x+\circ x=0.
$$
The second equation yiels that $x$ is quadatic, and the first one then yields that $y$ is in the same number field.
The subperiods finally yield that all the other Grassmann coordinates are also in this number field.\\
{\bf Subcase A3}: $q=0$, $pr\neq 0$.\\
Set $G_{24}=x$, $G_{13}=y$ and normalize to $G_{12}=1$.
The subperiods yield
$$
G_{23}=\circ y+\circ,\qquad
G_{14}=\circ,\qquad
G_{34}=\circ y+\circ.
$$
The Plücker and coincidence relations respectively become
$$
\circ xy+\circ y+\circ=0
\qquad\textrm{and}\qquad
\circ xy+\circ x+\circ y+\circ=0.
$$
We use the first equation to express $x$ as a function of $y$ and plug this into the second equation to get that $x$ is quadratic.
We then deduce that all the remaining Grassmann coordinates also belong to this quadratic field.\\
{\bf Subcase A4}: $r=0$, $pq\neq 0$.\\
Set $G_{34}=x$, $G_{12}=y$ and normalize to $G_{13}=1$.
The subperiods yield
$$
G_{23}=\circ y+\circ,\qquad
G_{24}=\circ y+\circ,\qquad
G_{14}=\circ.
$$
The Plücker and coincidence relations can be rewritten as in Subcase A3.
We then deduce that all the remaining Grassmann coordinates also belong to this quadractic field.\\
{\bf Subcase A5}: $p\neq 0$, $q=r=0$\\
Set $G_{24}=x$, $G_{34}=y$ and normalize to $G_{12}=1$.
The subperiods yield
$$
G_{23}=\circ,\qquad
G_{14}=\circ,\qquad
G_{13}=\circ.
$$
The Plücker and coincidence relations can be rewritten as in Subcase A3.
$$
\circ x+\circ y+\circ=0
\qquad\textrm{and}\qquad
\circ xy+\circ x+\circ y=0.
$$
We use the first equation to express $x$ as a function of $y$ and plug this into the second equation to get that $x$ is quadratic.
We then deduce that all the remaining Grassmann coordinates also belong to this quadractic field.\\
{\bf Subcase A6}: $q\neq 0$, $p=r=0$.\\
Set $G_{23}=x$, $G_{34}=y$ and normalize to $G_{12}=1$.
The subperiods yield
$$
G_{13}=\circ,\qquad
G_{24}=\circ,\qquad
G_{14}=\circ.
$$
The Plücker and coincidence relations respectively become
$$
\circ x+\circ y+\circ=0
\qquad\textrm{and}\qquad
\circ y+\circ=0.
$$
This yields that all the Grassmann coordinates are rational.
This subcase is excluded since the plane is irrationnal.\\
{\bf Subcase A7}: $r\neq 0$, $p=q=0$.\\
Set $G_{23}=x$, $G_{24}=y$ and normalize to $G_{12}=1$.
The subperiods yield
$$
G_{13}=\circ,\qquad
G_{14}=\circ,\qquad
G_{34}=\circ.
$$
The Plücker and coincidence relations can be rewritten as in Subcase A6.
This yields that all the Grassmann coordinates are rational, what is again excluded.\\
{\bf Subcase A8}: $p=q=r=0$.\\
The subperiods yield $G_{12}=0$: it is excluded because the plane is nondegenerate.

\paragraph{Grassmann coordinates are quadratic: case B.}~\\
The three subperiods are
\begin{eqnarray*}
pG_{23}&=&aG_{24}+bG_{34}\\
qG_{12}&=&cG_{14}+dG_{24}\\
rG_{13}&=&eG_{14}+fG_{34}.
\end{eqnarray*}
We make again eight subcases, but here, they eventually will be all excluded.\\
{\bf Subcase B1}: $pqr\neq 0$.\\
Set $G_{24}=x$, $G_{34}=y$ and normalize to $G_{14}=1$.
The subperiods yield
$$
G_{23}=\circ x+\circ y,\qquad
G_{12}=\circ x+\circ,\qquad
G_{13}=\circ y+\circ.
$$
The Plücker and coincidence relations both become
$$
\circ xy+\circ x+\circ y=0.
$$
We use the first equation to express $x$ as a function of $y$ and plug this into the second equation to get an algebraic fraction in $y$ whose numerator is $\circ y^2+\circ y$.
This yields that $y$, and then all the other Grassmann coordinates, are rational.
This is is excluded because the plane is irrationnal.\\
{\bf Subcase B2}: $p=0$, $qr\neq 0$.\\
Set $G_{23}=x$, $G_{14}=y$ and normalize to $G_{24}=1$.
The subperiods yield
$$
G_{34}=\circ,\qquad
G_{12}=\circ y+\circ,\qquad
G_{13}=\circ y+\circ.
$$
The Plücker and coincidence relations respectively become
$$
\circ xy+\circ y+\circ =0
\qquad\textrm{and}\qquad
\circ y+\circ=0.
$$
This yiels that $y$, and then all the other Grassmann coordinates, are rational: this is excluded.\\
{\bf Subcase B3}: $q=0$, $pr\neq 0$.\\
Set $G_{12}=x$, $G_{34}=y$ and normalize to $G_{24}=1$.
The subperiod yield
$$
G_{23}=\circ y+\circ,\qquad
G_{14}=\circ,\qquad
G_{13}=\circ y+\circ.
$$
The Plücker and coincidence relations both become
$$
\circ xy+\circ y+\circ =0.
$$
We use the first equation to express $x$ as a function of $y$ and plug this into the second equation to get an algebraic fraction in $y$ whose numerator is $\circ x+\circ$.
This yields that $x$, and then all the other Grassmann coordinates, are rational: this is excluded.\\
{\bf Subcase B4}: $r=0$, $pq\neq 0$.\\
Set $G_{13}=x$, $G_{24}=y$ and normalize $G_{14}=1$.
The subperiods yield
$$
G_{23}=\circ y+\circ,\qquad
G_{12}=\circ y+\circ,\qquad
G_{34}=\circ.
$$
The Plücker and coincidence relations can be rewritten as in Subcase B3.
This yields that all the Grassmann coordinates are rational: this is excluded.\\
{\bf Subcase B5}: $p\neq 0$, $q=r=0$.\\
Set $G_{12}=x$, $G_{13}=y$ and normalize to $G_{14}=1$.
The subperiods yield
$$
G_{23}=\circ,\qquad
G_{24}=\circ,\qquad
G_{34}=\circ.
$$
The Plücker and coincidence relations both become
$$
\circ x+\circ y+\circ 1=0.
$$
This yields that all the Grassmann coordinates are rational: this is excluded.\\
{\bf Subcase B6}: $q\neq 0$, $p=r=0$.\\
Set $G_{13}=x$, $G_{23}=y$ and normalize to $G_{24}=1$.
The subperiods yield
$$
G_{34}=\circ,\qquad
G_{12}=\circ,\qquad
G_{14}=\circ.
$$
The Plücker and coincidence relations respectively become
$$
\circ x+\circ y+\circ =0
\qquad\textrm{et}\qquad
\circ x+\circ=0.
$$
This yields that all the Grassmann coordinates are rational: this is excluded.\\
{\bf Subcase B7}: $r\neq 0$, $p=q=0$.\\
Set $G_{12}=x$, $G_{23}=y$ and normalize to $G_{24}=1$.
The subperiods yield
$$
G_{34}=\circ,\qquad
G_{14}=\circ,\qquad
G_{13}=\circ.
$$
The Plücker and coincidence relations can be rewritten as in Subcase B6.
This yields that all the Grassmann coordinates are rational: this is excluded.\\
{\bf Subcase B8}: $p=q=r=0$.\\
This subcase is excluded for the same reason as Subcase A8.

\paragraph{The four equations are independent.}\hfill\\
The only subcases to consider are A1 to A5.
Since the Grassmann coordinates are in the same quadratic field, there is a type $1$ subperiod:
$$
\alpha G_{23}+\beta G_{24}+\gamma G_{34}=0.
$$
{\bf Subcase A1}:\\
The Grassmann coordinates are
$$
G_{12}=y,\quad
G_{13}=x,\quad
G_{14}=1,\quad
G_{23}=ay+bx,\quad
G_{24}=cy+d,\quad
G_{34}=ex+f.
$$
The fourth subperiod becomes
$$
\alpha (ay+bx)+\beta (cy+d)+\gamma(ex+f)=0,
$$
that is,
$$
(a\alpha+c\beta)y+(b\alpha+e\gamma)x+(d\beta+f\gamma)=0.
$$
We use the Plücker relation to express $y$ as a function of $x$:
$$
y=\frac{(d-b)x}{(e-c)x+(a+f)}.
$$
By plugging this into the fourth subperiod, we get $Ax^2+Bx+C=0$ with
\begin{eqnarray*}
A&=&(e-c)(b\alpha+e\gamma),\\
B&=&(d-b)(a\alpha+c\beta)+(a+f)(b\alpha+e\gamma)+(e-c)(d\beta+f\gamma),\\
C&=&(a+f)(d\beta+f\gamma).
\end{eqnarray*}
In terms of matrices:
$$
\left(\hspace{-5pt}\begin{array}{ccc}
b(e-c) & 0 & e(e-c)\\
a(d-b)+b(a+f) & c(d-b)+d(e-c) & e(a+f)+f(e-c)\\
0 & d(a+f) & f(a+f)
\end{array}\hspace{-5pt}\right)
\hspace{-3pt}
\left(\hspace{-5pt}\begin{array}{c}
\alpha\\\beta\\\gamma
\end{array}\hspace{-5pt}\right)
\hspace{-3pt}=\hspace{-3pt}
\left(\hspace{-5pt}\begin{array}{c}
A\\B\\C
\end{array}\hspace{-5pt}\right).
$$
Since $x$ is irrational (otherwise all the other Grassmann coordinates, hence the plane, would be rational), the triple $(A,B,C)$ is unique (up to a common multiplicative factor).
The above matrix has thus non-zero determinant.
To compute it, we first factor $(e-c)$ in the firs line and $(a+f)$ in the third one.
We then substract from the second line $(a+f)$ times the first one and $(e-c)$ times the third one.
We can then factor the second line by $(d-b)$ and compute the determinant of the matrix
$$
\left(\begin{array}{ccc}
b & 0 & e\\
a & c & 0\\
0 & d & f
\end{array}\right).
$$
We finally get
$$
(e-c)(a+f)(d-b)(bcf+ade).
$$
Now, the matrix of the linear system formed by the four subperiods (the variables being the six Grassmann coordinates) is
$$
\left(\begin{array}{cccccc}
a & b & 0 & -p & 0 & 0\\
c & 0 & d & 0 & -q & 0\\
0 & e & f & 0 & 0 & -r\\
0 & 0 & 0 & \alpha & \beta & \gamma
\end{array}\right).
$$
In particular, the minor of size $4$ formed by the three first columns and one of the columns whose fourth entry is non-zero (at most one of the rational $\alpha$, $\beta$ and $\gamma$ can be equal to zero, since they are the coefficients of a subperiod) is proportional to $bcf+ade$.
It is thus non-zero {\em i.e.}, the four equations are independent.\\
{\bf Subcase A2}:\\
The Grassmann coordinates are
$$
G_{12}=1,\quad
G_{13}=-\frac{a}{b},\quad
G_{14}=x,\quad
G_{23}=y,\quad
G_{24}=c+dx,\quad
G_{34}=fx-\frac{ea}{b}.
$$
The fourth subperiod becomes
$$
\alpha y+\beta (c+dx)+\gamma(fx-ea/b)=0,
$$
that is,
$$
\alpha y+(d\beta+f\gamma)x+c\beta-\frac{ea}{b}\gamma=0.
$$
We use the Plücker relation to express $x$ as a function of $y$:
$$
x=\frac{a(e-c)}{bf+ad+by}.
$$
By plugging this into the fourth subperiod, we get $Ay^2+By+C=0$ with
\begin{eqnarray*}
A&=&b\alpha,\\
B&=&(bf+ad)\alpha+b(c\beta-\frac{ea}{b}\gamma),\\
C&=&(d\beta+f\gamma)a(e-c)+(bf+ad)(c\beta-\frac{ea}{b}\gamma).
\end{eqnarray*}
In terms of matrices:
$$
\left(\hspace{-5pt}\begin{array}{ccc}
b & 0 & 0\\
bf+ad & bc & -ea\\
0 & ad(e-c)+c(bf+ad) & af(e-c)-(bf+ad)\frac{ea}{b}
\end{array}\hspace{-5pt}\right)
\hspace{-3pt}
\left(\hspace{-5pt}\begin{array}{c}
\alpha\\\beta\\\gamma
\end{array}\hspace{-5pt}\right)
\hspace{-3pt}=\hspace{-3pt}
\left(\hspace{-5pt}\begin{array}{c}
A\\B\\C
\end{array}\hspace{-5pt}\right).
$$
The determinant of this matrix is non-zero.
It is equal to
$$
ab(e-c)(ade+bcf).
$$
This ensures as in Subcase A1 that the four equationss are independent.\\
{\bf Subcase A3}:\\
The Grassmann coordinates are
$$
G_{12}=1,\quad
G_{13}=y,\quad
G_{14}=-\frac{c}{d},\quad
G_{23}=a+by,\quad
G_{24}=x,\quad
G_{34}=ey-\frac{cf}{d}.
$$
The fourth subperiod becomes
$$
(b\alpha+e\gamma)y+\beta x+(a\alpha-\frac{cf}{d}\gamma)=0.
$$
We use the Plücker relation to express $y$ as a function of $x$:
$$
y=-\frac{c(a+f)}{dx-de+bc}.
$$
By plugging this into the fourth subperiod, we get $Ax^2+Bx+C=0$ with
\begin{eqnarray*}
A&=&d\beta\\
B&=&(bc-de)\beta+ad\alpha-cf\gamma\\
C&=&-(b\alpha+e\gamma)c(a+f)+(bc-de)(a\alpha-\frac{cf}{d}\gamma).
\end{eqnarray*}
In terms of matrices:
$$
\left(\hspace{-5pt}\begin{array}{ccc}
0 & d & 0\\
ad & bc-de & -cf\\
-(ade+bcf) & 0 & -\frac{c}{d}(ade+bcf)
\end{array}\hspace{-5pt}\right)
\hspace{-3pt}
\left(\hspace{-5pt}\begin{array}{c}
\alpha\\\beta\\\gamma
\end{array}\hspace{-5pt}\right)
\hspace{-3pt}=\hspace{-3pt}
\left(\hspace{-5pt}\begin{array}{c}
A\\B\\C
\end{array}\hspace{-5pt}\right).
$$
The determinant of this matrix is non-zero.
It is equal to
$$
cd(a+f)(ade+bcf).
$$
This ensures as in Subcase A1 that the four equations are independent.\\
{\bf Subcase A4}:\\
The Grassmann coordinates are
$$
G_{12}=y,\quad
G_{13}=1,\quad
G_{14}=-\frac{e}{f},\quad
G_{23}=ay+b,\quad
G_{24}=cy-\frac{ed}{f},\quad
G_{34}=x.
$$
The fourth subperiod becomes
$$
(a\alpha+c\beta)y+\gamma x+(b\alpha-\frac{ed}{f}\beta)=0.
$$
We use the Plücker relation to express $y$ as a function of $x$:
$$
y=\frac{e(b-d)}{fx-fc-ea}.
$$
By plugging this into the fourth subperiod, we get $Ax^2+Bx+C=0$ with
\begin{eqnarray*}
A&=&f\gamma\\
B&=&bf\alpha-ed\beta-(cf+ae)\gamma\\
C&=&-(ade+bcf)\alpha+\frac{e}{f}(ade+bcf)\beta
\end{eqnarray*}
In terms of matrices:
$$
\left(\hspace{-5pt}\begin{array}{ccc}
0 & 0 & f\\
bf & -ed & -cf-ae\\
-(ade+bcf) & \frac{e}{f}(ade+bcf) & 0
\end{array}\hspace{-5pt}\right)
\hspace{-3pt}
\left(\hspace{-5pt}\begin{array}{c}
\alpha\\\beta\\\gamma
\end{array}\hspace{-5pt}\right)
\hspace{-3pt}=\hspace{-3pt}
\left(\hspace{-5pt}\begin{array}{c}
A\\B\\C
\end{array}\hspace{-5pt}\right).
$$
The determinant of this matrix is non-zero.
It is equal to
$$
ef(b-d)(ade+bcf).
$$
This ensures as in Subcase A1 that the four equations are independent.\\
{\bf Subcase A5}:\\
The Grassmann coordinates are
$$
G_{12}=1,\quad
G_{13}=\frac{cf}{ed},\quad
G_{14}=-\frac{c}{d},\quad
G_{23}=a+\frac{bcf}{ed},\quad
G_{24}=x,\quad
G_{34}=y.
$$
In this case, the Plücker relation {\em is} the fourth subperiod:
$$
y=\frac{cf}{ed}x-\frac{c}{d}(a+\frac{bcf}{ed}).
$$
Note that 
$$
G_{23}=\frac{ade+bcf}{ed}.
$$
Hence, $ade+bcf\neq 0$ because the plane is non-degenerated.
This ensures as in Subcase A1 that the four equations are independent.
\end{proof}


\bibliographystyle{alpha}
\bibliography{quadratic}

\begin{thebibliography}{BF15b}

\bibitem[Bee82]{Beenker-1982}
F.~P.~M. Beenker.
\newblock Algebric theory of non periodic tilings of the plane by two simple
  building blocks: a square and a rhombus.
\newblock Technical Report TH Report 82-WSK-04, Technische Hogeschool
  Eindhoven, 1982.

\bibitem[BF13]{Bedaride-Fernique-2013}
N.~B{\'e}daride and Th. Fernique.
\newblock {\em Aperiodic Crystals}, chapter The {A}mmann--{B}eenker tilings
  revisited, pages 59--65.
\newblock Springer Netherlands, Dordrecht, 2013.

\bibitem[BF15a]{Bedaride-Fernique-2015b}
N.~B{\'e}daride and Th. Fernique.
\newblock No weak local rules for the 4p-fold tilings.
\newblock {\em Discrete \& Computational Geometry}, 54:980--992, 2015.

\bibitem[BF15b]{Bedaride-Fernique-2015}
N.~B{\'e}daride and Th. Fernique.
\newblock When periodicities enforce aperiodicity.
\newblock {\em Communications in Mathematical Physics}, 335:1099--1120, 2015.

\bibitem[BG13]{Baake-Grimm-2013}
M.~Baake and U.~Grimm.
\newblock {\em Aperiodic Order}, volume 149 of {\em Encyclopedia of mathematics
  and its applications}.
\newblock Cambridge University Press, 2013.

\bibitem[Bur88]{Burkov-1988}
S.~E. Burkov.
\newblock Absence of weak local rules for the planar quasicrystalline tiling
  with the {$8$}-fold rotational symmetry.
\newblock {\em Communications in Mathermatical Physics}, 119:667--675, 1988.

\bibitem[DB81]{deBruijn-1981}
N.~G. De~Bruijn.
\newblock Algebraic theory of {P}enrose's nonperiodic tilings of the plane.
\newblock {\em Nederl. Akad. Wetensch. Indag. Math.}, 43:39--52, 1981.

\bibitem[FO10]{Fernique-Ollinger-2010}
Th. Fernique and N.~Ollinger.
\newblock Combinatorial substitutions and sofic tilings.
\newblock In J.~Kari, editor, {\em JAC}, pages 100--110. Turku Center for
  Computer Science, 2010.

\bibitem[FS]{Fernique-Sablik-2016}
Th. Fernique and M.~Sablik.
\newblock Weak colored local rules for planar tilings.
\newblock arxiv:1603.09485.

\bibitem[GS98]{Goodman-Strauss-1998}
Ch. Goodman-Strauss.
\newblock Matching rules and substitution tilings.
\newblock {\em Annals of Mathematics}, 147:181--223, 1998.

\bibitem[HP94]{Hodge-Pedoe-1994}
W.~V.~D. Hodge and D.~Pedoe.
\newblock {\em Methods of Algebraic Geometry}, volume~2.
\newblock Cambridge University Press, 1994.

\bibitem[Jul10]{Julien-2010}
A.~Julien.
\newblock Complexity and cohomology for cut-and-projection tilings.
\newblock {\em Ergodic Theory and Dynamical Systems}, 30:489--523, 2010.

\bibitem[Kat88]{Katz-1988}
A.~Katz.
\newblock Theory of matching rules for the 3-dimensional {P}enrose tilings.
\newblock {\em Communications in Mathematical Physics}, 118:263--288, 1988.

\bibitem[Kat95]{Katz-1995}
A.~Katz.
\newblock {\em Beyond Quasicrystals: Les Houches, March 7--18, 1994}, chapter
  Matching rules and quasiperiodicity: the octagonal tilings, pages 141--189.
\newblock Springer Berlin Heidelberg, Berlin, Heidelberg, 1995.

\bibitem[KP87]{Kleman-Pavlovitch-1987}
M.~Kleman and A.~Pavlovitch.
\newblock Generalised 2d penrose tilings: structural properties.
\newblock {\em Journal of Physics A: Mathematical and General}, 20:687--702,
  1987.

\bibitem[Le92a]{Le-1992b}
T.~Q.~T. Le.
\newblock Local structure of quasiperiodic tilings having 8-fold symmetry.
\newblock 26 pages, preprint of the Max Planck Institute für Mathematik, 1992.

\bibitem[Le92b]{Le-1992c}
T.~Q.~T. Le.
\newblock Necessary conditions for the existence of local rules for
  quasicrystals.
\newblock 22 pages, preprint of the Max Planck Institute für Mathematik, 1992.

\bibitem[Le95]{Le-1995}
T.~Q.~T. Le.
\newblock Local rules for pentagonal quasi-crystals.
\newblock {\em Discrete \& Computational Geometry}, 14:31--70, 1995.

\bibitem[Le97]{Le-1997}
T.~Q.~T. Le.
\newblock Local rules for quasiperiodic tilings.
\newblock In {\em The mathematics of long-range aperiodic order ({W}aterloo,
  {ON}, 1997)}, volume 489 of {\em NATO Adv. Sci. Inst. Ser. C Math. Phys.
  Sci.}, pages 331--366. Kluwer Acad. Publ., Dordrecht, 1997.

\bibitem[Lev88]{Levitov-1988}
L.~S. Levitov.
\newblock Local rules for quasicrystals.
\newblock {\em Communications in mathematical physics}, 119:627--666, 1988.

\bibitem[LPS92]{Le-Piunikhin-Sadov-1992}
T.~Q.~T. Le, S.~Piunikhin, and V.~Sadov.
\newblock Local rules for quasiperiodic tilings of quadratic {$2$}-planes in
  {${\bf R}^4$}.
\newblock {\em Communications in mathematical physics}, 150:23--44, 1992.

\bibitem[LPS93]{Le-Piunikhin-Sadov-1993}
T.~Q.~T. Le, S.~A. Piunikhin, and V.~A. Sadov.
\newblock The geometry of quasicrystals.
\newblock {\em Russian Mathematical Surveys}, 48(1):37--100, 1993.

\bibitem[Moz89]{Mozes-1989}
S.~Mozes.
\newblock Tilings, substitution systems and dynamical systems generated by
  them.
\newblock {\em Journal d’Analyse Mathématique}, 53:139--186, 1989.

\bibitem[Sen95]{Senechal-1995}
M.~Senechal.
\newblock {\em Quasicrystals and geometry}.
\newblock Cambridge University Press, 1995.

\bibitem[Soc89]{Socolar-1989}
J.~E.~S. Socolar.
\newblock Simple octagonal and dodecagonal quasicrystals.
\newblock {\em Physical Review B}, 39:519--551, 1989.

\bibitem[Soc90]{Socolar-1990}
J.~E.~S. Socolar.
\newblock Weak matching rules for quasicrystals.
\newblock {\em Communications in Mathermatical Physics}, 129:599--619, 1990.

\end{thebibliography}
\end{document}